\documentclass[a4paper, 12pt]{article}

\usepackage[T2A]{fontenc}
\usepackage[utf8]{inputenc}
\usepackage[english,russian]{babel}
\usepackage[tbtags]{amsmath}
\usepackage{amsfonts,amssymb}
\usepackage{cite, color}

\usepackage{graphicx}

\newtheorem{theorem}{Теорема}

\newtheorem{corollary}{Следствие}

\def\R{{\mathbb R}}

\newcommand{\be}[1]{\begin{equation}\label{#1}}
\newcommand{\ee}{\end{equation}}
\newcommand{\pder}[2]{\frac{\partial \, #1}{\partial \, #2} }

\title{Лоренцева задача на группе Гейзенберга}
\author{И.А. Галяев, Ю.Л. Сачков}

\begin{document}

\maketitle

%\date{\today}

\section{Введение}
С точки зрения глобальной дифференциальной геометрии общая теория относительности описывается лоренцевой геометрией \cite{Muller, beem}.
Важной исследовательской задачей является сопоставление методов и результатов лоренцевой геометрии с методами и результатами римановой геометрии. Например, в лоренцевой геометрии информация может распространяться вдоль
кривых с векторами скорости из некоторого острого конуса. Естественной является
задача отыскания лоренцевых длиннейших, максимизирующих функционал типа длины
вдоль допустимых кривых. Поэтому важной задачей является описание лоренцевых длиннейших для всех пар точек, где вторая достижима из первой вдоль допустимой кривой.
Эта задача полностью исследована лишь в простейшем случае
левоинвариантной лоренцевой структуры в $\mathbb{R}^{n+1}$, для пространства Минковского $\mathbb{R}_1^{n+1}$ \cite{Ivanov}.\\

Левоинвариантная лоренцева структура на группе Ли - это невырожденная квадратичная форма $g$ индекса 1 на алгебре Ли $\mathfrak{g}$. Напомним некоторые основные определения лоренцевой геометрии \cite{intro}. Элемент $X\in \mathfrak{g}$ называется времениподобным, если $g(X) < 0$, пространственноподобным, если $g(X) > 0$, светоподобным (или нулевым), если $g(X)=0$. Липшицева кривая в $M$ называется времениподобной, если она имеет времениподобный вектор скорости почти везде; пространственноподобные, светоподобные и непространственноподобные кривые определяются аналогично. Ориентация времени $X_0$ - это произвольный времениподобный элемент $X_0\in \mathfrak{g}$. Направленная в будущее времениподобная кривая $q(t), t \in [0, t_1]$, называется параметризованной длиной дуги, если $g(\dot{q(t)},\dot{q(t)}) \equiv -1$. Любая направленная в будущее времениподобная кривая может быть параметризована длиной дуги, аналогично римановой геометрии.

Лоренцева длина непространственноподобной кривой $\gamma \in \mathrm{Lip}([0,t_1],G)$ это:
$$l(\gamma)=\int_0^{t_1}\sqrt{|g(\dot{\gamma},\dot{\gamma})|}dt.$$
Для точек $q_0, q_1 \in G$ обозначим через $\Omega_{q_0q_1}$ множество всех будущих направленных непространственноподобных кривых в $G$, которые соединяют $q_0$ с $q_1$. В случае $\Omega_{q_0q_1}=\emptyset$ определим лоренцево расстояние (функцию разделения времени) от точки $q_0$ до точки $q_1$ как
\begin{equation}\label{eq:dlinna}
    d(q_0,q_1)=\mathrm{sup}\{l(\gamma)|\gamma\in\Omega_{q_0q_1}\}.
\end{equation}
Если $\Omega_{q_0q_1}=\emptyset$, по определению $d(q_0,q_1)=0$.

Направленная в будущее непространственноподобная кривая $\gamma$ называется максимизатором длины Лоренца, если она максимальна в \eqref{eq:dlinna} между своими конечными точками $\gamma(0) = q_0, \gamma(t_1) = q_1$. Причинное будущее точки $q_0 \in G$ - это множество $J^+_{q_0}$ точек $q_1 \in G$, для которых существует направленная в будущее непространственноподобная кривая $\gamma$, которая соединяет $q_0$ и $q_1$. Причинное прошлое $J^-_{q_0}$ определяется аналогично в терминах направленных в прошлое непространственноподобных кривых.

Пусть $q_0 \in G, q_1 \in J^+_{q_0}$. Поиск максимизаторов длины Лоренца, которые соединяют $q_0$ с $q_1$, сводится к поиску будущих направленных непространственноподобных кривых $\gamma$, которые решают задачу
$$l(\gamma)\rightarrow \max, \ \ \ \gamma(0)=q_0, \ \ \ \gamma(t_1)=q_1.$$

В этой статье рассмотрено три лоренцевы задачи на группе Гейзенберга. К задачам применен принцип максимума Понтрягина, получена параметризация анормальных и нормальных экстремальных траекторий. Исследованы множества достижимости и существование оптимальных траекторий.

Группа Гейзенберга есть пространство $G \simeq\R^3 = \{(x,y, z) \mid x, y, z \in \R\}$ с базисом левоинвариантных векторных полей
$$
X_1 = \pder{}{x} - \frac y2 \pder{}{z}, \qquad
X_2 = \pder{}{y} + \frac x2 \pder{}{z}, \qquad
X_3 =   \pder{}{z}.
$$

\section{Первая задача}
\subsection{Постановка задачи}
Сформулируем первую задачу Лоренца на группе Гейзенберга следующим образом \cite{beem}:
\begin{align}
&\dot q = \sum_{i=1}^3 u_i X_i, \qquad q = (x, y, z) \in G, \label{pr1}\\
&u \in U = \{ (u_1, u_2, u_3) \in \R^3 \mid u_1^2 + u_2^2- u_3^2 \leq 0, \ u_3 \geq 0\},\label{pr2}\\
&q(0) = q_0 = (0, 0, 0), \qquad q(t_1) = q_1,  \label{pr3}\\
&J(\gamma) = \int_0^{t_1} \sqrt{u_3^2 - u_1^2-u_2^2} \, dt  \to \max. \label{pr4}
\end{align}
%$$g(u)=u_1^2 + u_2^2- u_3^2$$

%Пространство левоинвариантных лоренцевых задач параметризуется матрицей
%\begin{equation*}
%A=\begin{array}{l}
%\begin{pmatrix}
%1 & 0 & 0 \\
%0 & 1 & 0 \\
%0 & 0 & 1 \\
%\end{pmatrix}.
%\end{array}
%\end{equation*}

\noindent Система управления \eqref{pr1} записывается в координатах следующим образом:
\begin{equation}\label{eq:sys}
%\begin{cases}
\left\{\begin{array}{l}
\displaystyle \dot{x} = u_1, \\
\displaystyle \dot{y} = u_2, \\
\displaystyle \dot{z} = -\frac{y}{2}u_1+\frac{x}{2}u_2+u_3.
\end{array}\right.
\end{equation}

Запишем общий вид функции Понтрягина
\begin{equation*}\label{eq:pontHam}
h_{u}^\nu =\langle p, \sum\limits_{i=1}^3 u_i X_i\rangle + \nu  \sqrt{u_3^2 - u_1^2-u_2^2}, \quad p \in T^*M, \,\, \nu \leq 0.
\end{equation*}

Пусть процесс $(u(t), q(t))$, $t \in[0,T]$, будет оптимальным, то выполняются следующие условия:
\begin{enumerate}
\item Гамильтонова система $\displaystyle \dot{p}=-\frac{\partial h_u^\nu}{\partial q},\,\,  \dot{q}=\frac{\partial h_u^\nu}{\partial p}$;
\item Условие максимума $h_{u(t)}^\nu (p(t),q(t)) ={\underset{u \in \mathbb{R}^3}{\max}}\, h_{u}^\nu (p(t),q(t)) $;
\item Условие нетривиальности $(p(t),\nu)\neq(0,0)$ $\forall t \in [0,T]$.
\end{enumerate}

Обозначим $h_i=\langle p,X_i\rangle.$ Тогда функция Понтрягина выражается следующим образом:
\begin{equation*}
h_{u}^\nu =u_1 h_1+u_2 h_2 +u_3 h_3-\nu  \sqrt{u_3^2 - u_1^2-u_2^2}.
\end{equation*}

В формулировке ПМП, не ограничивая общности, достаточно рассмотреть два случая: $\nu= 0$ — аномальный случай и $\nu =-1$ --- нормальный случай. Рассмотрим их подробно.

\subsection{Аномальный случай принципа максимума Понтрягина}
Пусть $\nu=0$. Определим $(a,b,c)$ как компоненты ковектора $p$ в канонических координатах. Они изменяются по закону $\displaystyle \dot{p_i}=-\frac{\partial h_{u}^0}{\partial q_i}$:

Систему управления можно записать в виде: \\
\begin{equation*}
\left\{\begin{array}{l}
\displaystyle \dot{x} = u_1, \\
\displaystyle \dot{y} = u_2, \\
\displaystyle \dot{z} = -\frac{y}{2}u_1+\frac{x}{2}u_2+u_3,
\end{array}\right.
\quad
\left\{\begin{array}{l}
\displaystyle \dot{a} = -c\frac{u_2}{2}, \\
\displaystyle \dot{b} = c\frac{u_1}{2}, \\
\displaystyle \dot{c} = 0.
\end{array}\right.
\end{equation*}\\ 

Связь между $h_i$ и $p_i$ явно записывается как: \\
\begin{equation*}
%\begin{cases}
\left\{\begin{array}{l}
\displaystyle h_1 = a-c\frac{y}{2}, \\
\displaystyle h_2 = b+c\frac{x}{2}, \\
\displaystyle h_3 = c. \\
\end{array}\right.
%\end{cases}
\end{equation*}\\ 

Функция Понтрягина $h_{u}^0$ принимает вид $h_{u}^0 =u_1 h_1+u_2 h_2+u_3 h_3.$ 
Выясним, при каких условиях функция $h_{u}^0$  достигает максимума, и на каких управлениях  $u \in U$.\\

$h_i$ суть наперед заданные константы. Опишем пространство $h_i$, разделив его на 4 подпространства.\\
1) $h_1^2+h_2^2-h_3^2 \leq 0$, $h_3 \geq 0$.\\
Запишем ограничения для случая 1
\begin{equation*}
\left\{\begin{array}{l}
\displaystyle u_3^2 \geq u_1^2+u_2^2, \\
\displaystyle h_3^2 \geq h_1^2+h_2^2, \\
\end{array}\right.
\end{equation*}\\  
Тогда справедливы следующие неравенства\\
$u_3^2h_3^2\geq u_1^2h_1^2+h_2^2u_2^2+u_1^2h_2^2+u_2^2h_1^2, $\\
$u_1^2h_1^2+h_2^2u_2^2+u_1^2h_2^2+u_2^2h_1^2 - (u_1h_1+u_2h_2)^2=(u_1h_2-u_2h_1)^2\geq 0.$\\
Откуда следует, что \\
$|u_3h_3|\geq |u_1h_1|+|u_2h_2|.$\\
Исследуем $u_3h_3$. Если $h_3>0$, в совокупности с неограниченностью $u_3$, то и $u_3h_3$ неограниченна.\\
$u_3=0$ означает и $u_1=u_2=0$.\\
Максимума функции $h_u^0$ либо не существует, либо решение тривиально.\\

2) $h_1^2+h_2^2-h_3^2 = 0$, $h_3 < 0$.\\
Согласно предыдущему пункту, $|u_3h_3|\geq |u_1h_1|+|u_2h_2|.$ Тогда:\\
$h_u^0\leq 0.$\\
$h_u^0$ достигает максимума, когда $u_3^2=u_1^2+u_2^2.$ Тривиальный случай дает максимум $h_u^0= 0.$ Рассмотрим и другие случаи достижения максимума. 

3) $h_1^2+h_2^2-h_3^2 < 0$, $h_3 < 0.$\\
Почти аналогично 1 случаю. 
\begin{equation*}
\left\{\begin{array}{l}
\displaystyle u_3^2 \geq u_1^2+u_2^2, \\
\displaystyle h_3^2 \geq h_1^2+h_2^2, \\
\end{array}\right.
\end{equation*}\\ 
Только из-за строгого равенства $h_u^0 \neq 0.$ То есть в любом случае функция отрицательна. Максимума нет.

4) $h_1^2+h_2^2-h_3^2 > 0$.\\
При $u = k(h_1, h_2, \sqrt{h_1^2+h_2^2})$ получаем $h_u = k \sqrt{h_1^2+h_2^2}(\sqrt{h_1^2+h_2^2} + h_3) \to + \infty$  при $k \to + \infty$.\\
Поэтому в случае 4) максимума не существует.

Во 2 случае $u_3^2=u_1^2+u_2^2.$ Выбираем натуральную параметризацию $u_1^2+u_2^2=1$.
Максимум функции $h_u^0=h_1u_1+ h_2u_2 + h_3u_3=0$, когда она равна 0. \\
$h_1u_1+ h_2u_2 - \sqrt{h_1^2+h_2^2}\sqrt{u_1^2+u_2^2}=0.$\\
$\displaystyle u_1=\frac{h_1}{\sqrt{h_1^2+h_2^2}}$,$\displaystyle u_2=\frac{h_2}{\sqrt{h_1^2+h_2^2}}$.
\begin{equation*}
\left\{\begin{array}{l}
\displaystyle \dot{x} = \frac{h_1}{\sqrt{h_1^2+h_2^2}}, \\
\displaystyle \dot{y} = \frac{h_2}{\sqrt{h_1^2+h_2^2}}, \\
\displaystyle \dot{z} = -\frac{yh_1}{\sqrt{h_1^2+h_2^2}}+\frac{xh_2}{\sqrt{h_1^2+h_2^2}}+1,
\end{array}\right.
\quad
\left\{\begin{array}{l}
\displaystyle \dot{h_1} = -h_3h_2, \\
\displaystyle \dot{h_2} = h_3h_1, \\
\displaystyle \dot{h_3} = 0.
\end{array}\right.
\end{equation*}\\ 

Положим $h_1^2+h_2^2=1, h_3=1.$ Перейдем к полярной координате $h_1 = \cos \theta$, $h_2 = \sin \theta,$
$\theta = h_3t+\theta_0.$ Тогда решение системы:

\begin{equation*}
\left\{\begin{array}{l}
\displaystyle x = \sin (t-\theta_0)+\sin (\theta_0),\\
\displaystyle y = \cos (t-\theta_0)-\cos (\theta_0),\\
\displaystyle z= \frac{t+\sin t}{2}.\\
\end{array}\right.
\end{equation*}\\ 
\begin{figure}[h]
\includegraphics[width=0.5\linewidth, height=6cm]{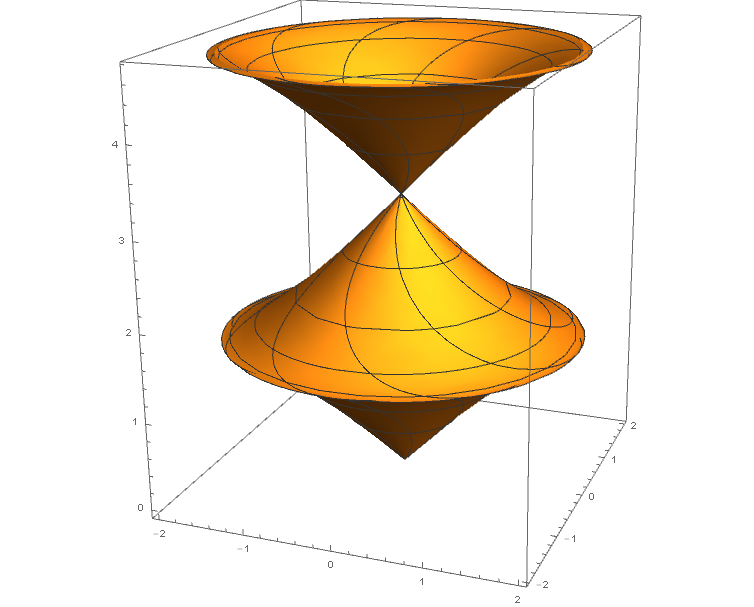} 
\includegraphics[width=0.5\linewidth, height=6cm]{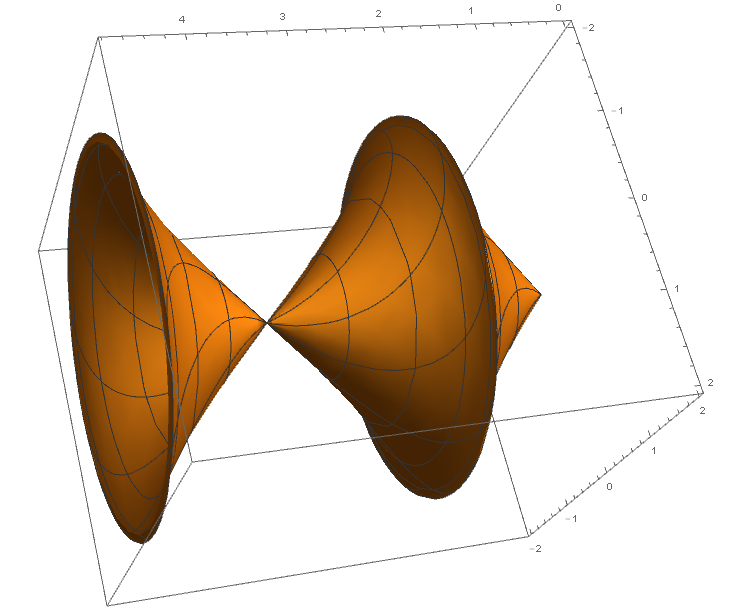}
\caption{График динамики системы при $\theta_0\in [-2\pi, 2\pi], t \in [0,10]$.}
\label{fig:image1}
\end{figure}

\subsection{Нормальный случай принципа максимума Понтрягина}
Пусть $\nu=-1$. Выпишем функцию Понтрягина для нормального случая:
\begin{equation*}
\displaystyle h_u = h_{u}^{-1} =u_1 h_1+u_2 h_2+u_3 h_3 + \sqrt{u_3^2 - u_1^2-u_2^2}.
\end{equation*}

Рассмотрим несколько вариантов\\
1) $h_1^2+h_2^2-h_3^2 \leq 0$, $h_3 \geq 0$.\\
Если выбрать $(u_1, u_2, u_3) = k(h_1, h_2, h_3)$, то
$$
h_u^{-1} \to + \infty, \text{ при } k \to \infty. 
$$
Поэтому в случае 1) максимума не существует.

2) $h_1^2+h_2^2-h_3^2 > 0$.\\
При $u = k(h_1, h_2, \sqrt{h_1^2+h_2^2})$ получаем \\
$h_u = k \sqrt{h_1^2+h_2^2}(\sqrt{h_1^2+h_2^2} + h_3) \to + \infty$  при $k \to + \infty$.\\
Поэтому в случае 2) максимума не существует.

2) $h_1^2+h_2^2-h_3^2 > 0$.\\
При $u = k(h_1, h_2, h_3)$ получаем $h_u = k \sqrt{h_1^2+h_2^2}(\sqrt{h_1^2+h_2^2} + h_3) \to + \infty$  при $k \to + \infty$.\\
Поэтому в случае 2) максимума не существует.

3) $h_1^2+h_2^2-h_3^2 = 0$, $h_3 < 0.$\\
Положим
\begin{align*}
&u_3 = \rho \cosh \alpha, \qquad u_1 = \rho \sinh \alpha \cos \beta, \qquad u_2 = \rho \sinh \alpha \sin \beta,\\
&h_3 = -R, \qquad h_1 = R  \cos b, \qquad h_2 = R   \sin b.
\end{align*}
Тогда $h_u = \rho(R(\sinh \alpha \cos(\beta-b)-\cosh \alpha)+1) \to + \infty$ при $\beta = b$, $\alpha\to +\infty$, $\rho \to + \infty$.\\
Поэтому в случае 3) максимума не существует.

4) $h_1^2+h_2^2-h_3^2 < 0$, $h_3 < 0.$\\
Положим
\begin{align*}
&u_3 = \rho \cosh \alpha, \qquad u_1 = \rho \sinh \alpha \cos \beta, \qquad u_2 = \rho \sinh \alpha \sin \beta,\\
&h_3 = -R \cosh a, \qquad h_1 = R  \sinh a \cos b, \qquad h_2 = R \sinh a  \sin b.
\end{align*}
Тогда $h_u = \rho(R(\sinh \alpha \sinh a \cos(\beta-b)-\cosh \alpha \cosh a)+1)$.

Если $R < 1$, то $h_u \to + \infty$  при $\beta = b$, $\alpha = a$, $h_u=\rho(1-R),\rho \to +\infty$.

Если $R > 1$, то $\max h_u =0$  при $\rho = 0$. 

Если $R = 1$, то $\max h_u = \sqrt{h_1^2+h_2^2-h_3^2} + 1$  при 
$$(u_1, u_2, u_3) = (h_1, h_2, -h_3)/ \sqrt{h_1^2+h_2^2-h_3^2}.$$

Поэтому в нормальном случае экстремали суть траектории гамильтоновой системы с гамильтонианом $H = (h_1^2+h_2^2-h_3^2)/2$:

\begin{align}
&\dot h_1 = - h_2 h_3, \label{dh1}\\
&\dot h_2 =  h_1 h_3, \label{dh2}\\
&\dot   h_3 = 0, \label{dh3}\\
&\dot q = h_1 X_1 + h_2 X_2 - h_3 X_3 \label{dq}.
\end{align}

Без ограничений общности задаем, что $h_1^2+h_2^2-h_3^2=1$. Тогда $H=1/2.$

\begin{equation*}
\left\{\begin{array}{l}
\displaystyle \dot{x} = h_1, \\
\displaystyle \dot{y} = h_2, \\
\displaystyle \dot{z} = -yh_1/2+xh_2/2-h_3,
\end{array}\right.
\quad
\left\{\begin{array}{l}
\displaystyle \dot{h_1} = - h_2 h_3, \\
\displaystyle \dot{h_2} = h_1 h_3, \\
\displaystyle \dot{h_3} = 0.
\end{array}\right.
\end{equation*}

Положим:
\begin{align*}
&h_3 = -\cosh a, \qquad h_1 =   \sinh a \cos \theta, \qquad h_2 =  \sinh a  \sin \theta.
\end{align*}

Тогда подсистема для $h_i$ примет вид:
\begin{equation*}
\left\{\begin{array}{l}
\displaystyle \dot{\theta} = -\cosh a, \\
\displaystyle \dot{h_3} = 0.
\end{array}\right.
\end{equation*}\\ 

Отсюда:
\begin{equation*}
\left\{\begin{array}{l}
\displaystyle \theta = \theta_0 -t \cosh a, \\
\displaystyle h_1 =  \sinh a \cos (\theta_0 -t \cosh a), \\
\displaystyle h_2 =  \sinh a \sin (\theta_0 -t \cosh a), \\
\displaystyle h_3 = - \cosh a.
\end{array}\right.
\end{equation*}\\ 
\begin{equation*}
\left\{\begin{array}{l}
\displaystyle x = \th a (\sin \theta_0-\sin (\theta_0 -t \cosh a)), \\
\displaystyle y = \th a (\cos (\theta_0 -t \cosh a)-\cos \theta_0), \\
\displaystyle z = t\frac{1+ \cosh^2 a}{2 \cosh a}+\frac{\th^2 a \sin (t \cosh a)}{2}.
\end{array}\right.
\end{equation*}\\ 
\begin{figure}[h]
\includegraphics[width=0.4\linewidth, height=6cm]{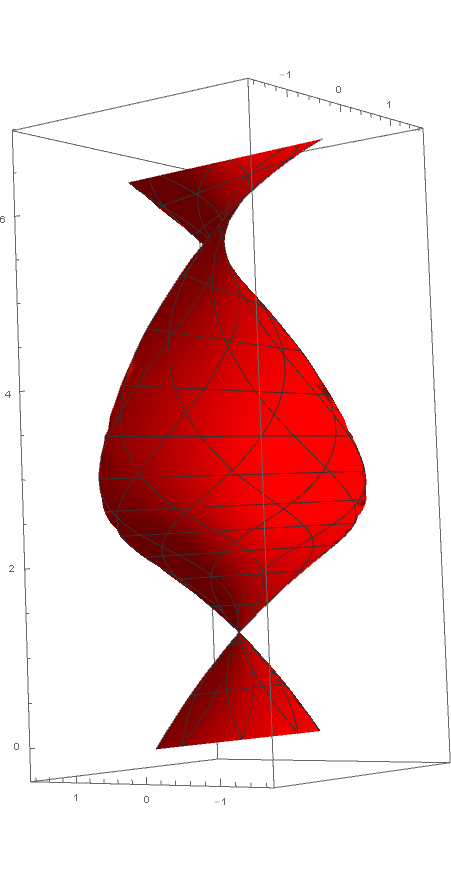} 
\includegraphics[width=0.55\linewidth, height=6cm]{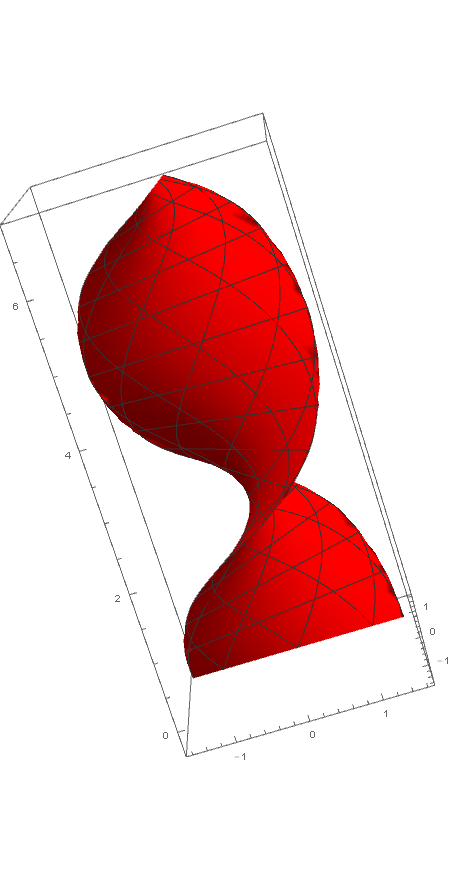}
\caption{График динамики системы при $a=1, \theta_0\in [-2\pi, 2\pi], t \in [0,6]$.}
\label{fig:image2}
\end{figure}

\subsection{Множество достижимости и существование оптимальных траекторий}
\begin{theorem}
Система \eqref{pr1}, \eqref{pr2} глобально управляема.
\end{theorem}
\begin{proofoftheorem}
Следует из принципа максимума Понтрягина в геометрической формулировке \cite{lor_lob}.
\end{proofoftheorem}

\begin{corollary}
Для любой точки $q \in M$ существует замкнутая допустимая траектория системы \eqref{pr1}, \eqref{pr2} положительной длины \eqref{pr4}.
\end{corollary}

\begin{corollary}
Для любых точек $q_0,q_1 \in M$ лоренцево расстояние между ними есть $d(q_0,q_1)=+\infty.$ Поэтому не существует лоренцевой длиннейшей, их соединяющей.
\end{corollary}

\section{Вторая задача}
\subsection{Постановка задачи}
Сформулируем вторую задачу Лоренца на группе Гейзенберга следующим образом \cite{beem}:
\begin{align}
&\dot q = \sum_{i=1}^3 u_i X_i, \qquad q = (x, y, z) \in G, \label{pr21}\\
&u \in U = \{ (u_1, u_2, u_3) \in \R^3 \mid -u_1^2 + u_2^2+ u_3^2 \leq 0, \ u_1 \geq 0\},\label{pr22}\\
&q(0) = q_0 = (0, 0, 0), \qquad q(t_1) = q_1,  \label{pr23}\\
&J(\gamma) = \int_0^{t_1} \sqrt{u_1^2 - u_2^2-u_3^2} \, dt  \to \max. \label{pr24}
\end{align}
Обозначим $h_i=\langle p,X_i\rangle.$ Тогда функция Понтрягина выражается следующим образом:
\begin{equation*}
h_{u}^\nu =u_1 h_1+u_2 h_2 +u_3 h_3-\nu  \sqrt{u_1^2 - u_2^2-u_3^2}.
\end{equation*}

\subsection{Аномальный случай принципа максимума Понтрягина}
Пусть $\nu=0$.
Определим $(a,b,c)$ как компоненты ковектора $p$ в канонических координатах. Они изменяются по закону $\displaystyle \dot{p_i}=-\frac{\partial h_{u}^0}{\partial q_i}$:

Систему управления можно записать в виде: \\
\begin{equation*}
\left\{\begin{array}{l}
\displaystyle \dot{x} = u_1, \\
\displaystyle \dot{y} = u_2, \\
\displaystyle \dot{z} = -\frac{y}{2}u_1+\frac{x}{2}u_2+u_3,
\end{array}\right.
\quad
\left\{\begin{array}{l}
\displaystyle \dot{a} = -c\frac{u_2}{2}, \\
\displaystyle \dot{b} = c\frac{u_1}{2}, \\
\displaystyle \dot{c} = 0.
\end{array}\right.
\end{equation*}\\ 

Связь между $h_i$ и $p_i$ явно записывается как: \\
\begin{equation*}
%\begin{cases}
\left\{\begin{array}{l}
\displaystyle h_1 = a-c\frac{y}{2}, \\
\displaystyle h_2 = b+c\frac{x}{2}, \\
\displaystyle h_3 = c. \\
\end{array}\right.
%\end{cases}
\end{equation*}\\ 

Функция Понтрягина $h_{u}^0$ принимает вид $h_{u}^0 =u_1 h_1+u_2 h_2+u_3 h_3.$ 
Выясним, при каких условиях функция $h_{u}^0$  достигает максимума, и на каких управлениях  $u \in U$.\\

$h_i$ суть наперед заданные константы. Опишем пространство $h_i$, разделив его на 4 подпространства.\\
1) $h_2^2+h_3^2-h_1^2 \leq 0$, $h_1 \geq 0$.\\
Запишем ограничения для случая 1
\begin{equation*}
\left\{\begin{array}{l}
\displaystyle u_1^2 \geq u_2^2+u_3^2, \\
\displaystyle h_1^2 \geq h_2^2+h_3^2, \\
\end{array}\right.
\end{equation*}\\  
Тогда справедливы следующие неравенства\\
$u_1^2h_1^2\geq u_2^2h_2^2+h_3^2u_3^2+u_2^2h_3^2+u_3^2h_2^2, $\\
$u_2^2h_2^2+h_3^2u_3^2+u_2^2h_3^2+u_3^2h_2^2 - (u_2h_2+u_3h_3)^2=(u_2h_3-u_3h_2)^2\geq 0.$\\
Откуда следует, что \\
$|u_1h_1|\geq |u_2h_2|+|u_3h_3|.$\\
Исследуем $u_1h_1$. Если $h_1>0$, в совокупности с неограниченностью $u_1$, то и $u_1h_1$ неограниченна.\\
$u_1=0$ означает и $u_2=u_3=0$.\\
Максимума функции $h_u^0$ либо не существует, либо решение тривиально.\\

2) $h_2^2+h_3^2-h_1^2 < 0$, $h_1 < 0.$\\
Почти аналогично 1 случаю. 
\begin{equation*}
\left\{\begin{array}{l}
\displaystyle u_1^2 \geq u_2^2+u_3^2, \\
\displaystyle h_1^2 \geq h_2^2+h_3^2, \\
\end{array}\right.
\end{equation*}\\ 
Только из-за строгого равенства $h_u^0 \neq 0.$ То есть в любом случае функция отрицательна. Максимума нет.

3) $h_2^2+h_3^2-h_1^2 > 0$.\\
При $u = k(\sqrt{h_2^2+h_3^2},h_2, h_3)$ получаем $h_u = k \sqrt{h_2^2+h_3^2}(\sqrt{h_2^2+h_3^2} + h_1) \to + \infty$  при $k \to + \infty$.\\
Поэтому в случае 3) максимума не существует.

4) $h_2^2+h_3^2-h_1^2 = 0$, $h_1 < 0$.\\
$u_1^2=u_2^2+u_3^2.$ Выбираем натуральную параметризацию $u_3=-1$, $u_1^2=u_2^2+1$.\\
Максимум функции $h_u^0=h_1u_1+ h_2u_2 + h_3u_3=0$, когда она равна 0. \\
$h_1^2 = h_2^2+h_3^2$. $h_1^2 = h_2^2+1,$ при этом $h_3=-1.$\\
$h_1u_1+ h_2u_2 + \sqrt{h_1^2-h_2^2}\sqrt{u_1^2-u_2^2}=0.$

В анормальном случае управление имеет вид $\displaystyle u=(\sqrt{h_2^2+h_3^2},h_2,h_3),$ поэтому анормальные траектории светоподобны.

Пусть $h_1^0=-\sqrt{(h_2^0)^2+h_3^2}.$ Тогда анормальные экстремали имеют следующий вид.\\
Если $h_3=0,$ то $h_1,h_2 \equiv const$ и $x=-h_1^0t, y=h_2^0t, z=0.$\\
Если $h_3\neq0$, то

\begin{equation*}
\left\{\begin{array}{l}
\displaystyle h_1 = |h_3|(\cosh C -\cosh \tau)+h_1^0,\\
\displaystyle h_2 = h_3\sinh \tau,\\
\displaystyle h_3 \equiv \mathrm{const},\\
\displaystyle x = \sinh \tau -\sinh C,\\
\displaystyle y = \mathrm{sgn} h_3(\cosh \tau -\cosh C),\\
\displaystyle z = \frac{h_3t+\sinh(h_3t))}{2},\\
\displaystyle C=\mathrm{arcsinh}\frac{h_2^0}{h_3}, \ \ \ \ \tau = C+|h_3|t.
\end{array}\right.
\end{equation*}

\begin{figure}
	    \centering
	    \includegraphics[width = 10cm]{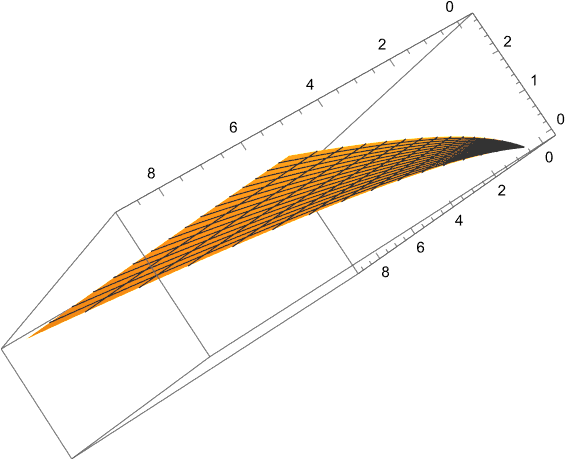}
            \caption{Семейство решений системы с переменными $t\in[0,20], C_1 \in[0,1]$ и зафиксированной константой $h_3=0.1$.}
\end{figure}

\subsection{Нормальный случай принципа максимума Понтрягина}
Пусть $\nu=-1$. Выпишем функцию Понтрягина для нормального случая:
\begin{equation*}
\displaystyle h_u = h_{u}^{-1} =u_1 h_1+u_2 h_2+u_3 h_3 + \sqrt{u_1^2 - u_2^2-u_3^2}.
\end{equation*}

Рассмотрим несколько вариантов\\
1) $h_2^2+h_3^2-h_1^2 \leq 0$, $h_1 \geq 0$.\\
Если выбрать $(u_1, u_2, u_3) = k(h_1, h_2, h_3)$, то
$$
h_u^{-1} \to + \infty, \text{ при } k \to \infty. 
$$
Поэтому в случае 1) максимума не существует.

2) $h_2^2+h_3^2-h_1^2 > 0$.\\
При $u = k(\sqrt{h_2^2+h_3^2}, h_2, h_3)$ получаем \\
$\displaystyle h_u = k \sqrt{h_2^2+h_3^2}(h_1 + \sqrt{h_2^2+h_3^2}) \to + \infty$  при $k \to + \infty$.\\
Поэтому в случае 2) максимума не существует.

3) $h_2^2+h_3^2-h_1^2 = 0$, $h_1 < 0.$\\
Положим
\begin{align*}
&u_1 = \rho \cosh \alpha, \qquad u_2 = -\rho \sinh \alpha \cos \beta, \qquad u_3 = -\rho \sinh \alpha \sin \beta,\\
&h_1 = -R  \cosh b, \qquad h_2 = -R   \sinh b, \qquad h_3 = -R.
\end{align*}
Тогда $h_u = \rho(R(-\cosh \alpha \cosh b+\sinh \alpha \sinh b \cos \beta+\sinh \alpha \sin \beta)+1) \to + \infty$ при $b=0$, $b \to + \infty$, $\alpha\to +\infty$, $\rho \to + \infty$.\\
Поэтому в случае 3) максимума не существует.

4) $h_2^2+h_3^2-h_1^2 < 0$, $h_1 < 0.$\\
Положим
\begin{align*}
&u_1 = \rho \cosh \alpha, \qquad u_2 = \rho \sinh \alpha \cos \beta, \qquad u_3 = \rho \sinh \alpha \sin \beta,\\
&h_1 = -R  \cosh b, \qquad h_2 = R   \sinh b, \qquad h_3 = R.
\end{align*}
Тогда $h_u = \rho(R(-\cosh \alpha \cosh b+\sinh \alpha \sinh b \cos \beta+\sinh \alpha \sin \beta)+1) \to + \infty$ при $b=0$, $b \to + \infty$, $\alpha\to +\infty$, $\rho \to + \infty$.\\

Если $R < 1$, то $h_u \to + \infty$  при $\beta=0$, $b = \alpha$, $h_u=\rho(1-R).$ $\rho \to +\infty$.

Если $R > 1$, то $\max h_u =0$  при $\rho = 0$.

Если $R = 1$, то $\max h_u = \sqrt{h_2^2+h_3^2-h_1^2} + 1$ при 
$$(u_1, u_2, u_3) = (-h_1, h_2, h_3)/ \sqrt{h_2^2+h_3^2-h_1^2}.$$

Поэтому в нормальном случае экстремали суть траектории гамильтоновой системы с гамильтонианом $H = (h_1^2+h_2^2-h_3^2)/2$:
\begin{align}
&\dot h_1 = - h_2 h_3, \label{dh21}\\
&\dot h_2 = - h_1 h_3, \label{dh22}\\
&\dot   h_3 = 0, \label{dh23}\\
&\dot q = -h_1 X_1 + h_2 X_2 + h_3 X_3 \label{d2q}.
\end{align}

Без ограничений общности задаем, что $h_2^2+h_3^2-h_1^2=1$. Тогда $H=1/2.$

\begin{equation*}
\left\{\begin{array}{l}
\displaystyle \dot{x} = -h_1, \\
\displaystyle \dot{y} = h_2, \\
\displaystyle \dot{z} = yh_1/2+xh_2/2+h_3,
\end{array}\right.
\quad
\left\{\begin{array}{l}
\displaystyle \dot{h_1} = - h_2 h_3, \\
\displaystyle \dot{h_2} = -h_1 h_3, \\
\displaystyle \dot{h_3} = 0.
\end{array}\right.
\end{equation*}

Если $h_3=0,$ то $h_1,h_2\equiv \mathrm{const}$ и $x=-h_1t, y=h_2t, z=0.$
Если $h_3\neq0,$ то решение системы

\begin{equation*}
\left\{\begin{array}{l}
\displaystyle x = \frac{h_2^0(\cosh s-1)-h_1^0\sinh s}{h_3}, \\
\displaystyle y = \frac{h_2^0\sinh s-h_1^0(\cosh s-1)}{h_3}, \\
\displaystyle z = \frac{((2h_3^2-(h_1^0)^2+(h_2^0)^2)s+((h_1^0)^2-(h_2^0)^2)\sinh s)}{2h_3^2},\\
\displaystyle s=h_3t.
\end{array}\right.
\end{equation*}

\begin{figure}[h]
	    \centering
	    \includegraphics[width = 5.5cm]{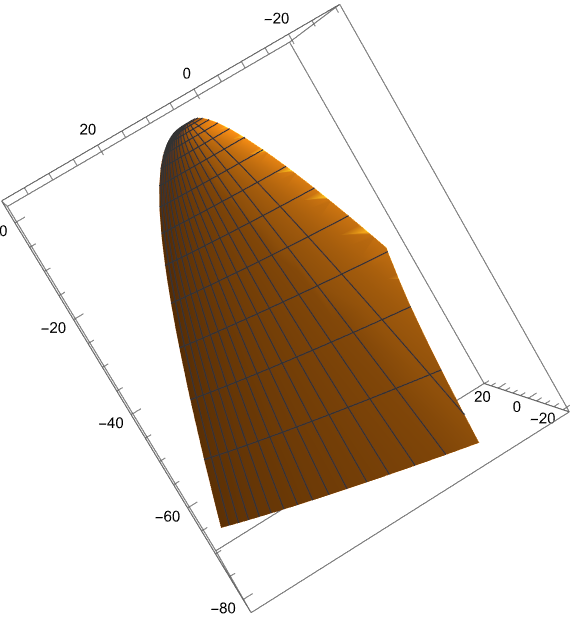}
            \caption{Семейство решений системы с переменными $t\in[0,20], h_1 \in[0,1]$ и зафиксированной константой $h_3=0.1$.}
\end{figure}

\subsection{Множество достижимости и существование оптимальных траекторий}
\begin{theorem}
Множество достижимости системы \eqref{pr21}, \eqref{pr22} из точки $q_0$ за произвольное неотрицательное время есть
$\mathcal{A}_{q_0}=$\\
$\{(x,y,z)\in M |0 < |z|\leq \frac{t+\sinh t}{2},t=\mathrm{arcosh}\frac{x^2-y^2+2}{2}\}\cup \{(x,y,z)\in M |x\geq|y|, z=0\}.$
\end{theorem}
\begin{proofoftheorem}
Следует из теоремы Адамара о глобальном диффеоморфизме.
\end{proofoftheorem}

\begin{theorem}
Для любых точек $q_0 \in M, q_1\in \mathcal{A}_0$ существует лоренцева длиннейшая, их соединяющей.
\end{theorem}
\begin{proofoftheorem}
Следует из теоремы о существовании лоренцевых длиннейших в глобально гиперболических лоренцевых многообразиях \cite{beem}.
\end{proofoftheorem}

\section{Третья задача}
\subsection{Постановка задачи}
Сформулируем третью Лоренцеву задачу на группе Гейзенберга следующим образом \cite{beem}:
\begin{align}
&\dot q = \sum_{i=1}^3 u_i Y_i, \qquad q = (x, y, z) \in G, \label{pr31}\\
&u \in U = \{ (u_1, u_2, u_3) \in \R^3 \mid u_1^2 - u_2^2- u_3^2 \geq 0, \ u_1 \geq 0\},\label{pr32}\\
&q(0) = q_0 = (0, 0, 0), \qquad q(t_1) = q_1,  \label{pr33}\\
&J = \int_0^{t_1} \sqrt{u_1^2 - u_2^2-u_3^2} \, dt  \to \max, \label{pr34}
\end{align}
\noindent с базисом левоинвариантных векторных полей:
\begin{align}
&Y_1 = \pder{}{x} + \pder{}{y} + \frac{x-y}{2}\pder{}{z}, \qquad
Y_2 = \pder{}{x} - \pder{}{y} - \frac{x+y}{2}\pder{}{z}, \\
&Y_3 =   -\pder{}{x} - \pder{}{y} + \frac{2-x+y}{2}\pder{}{z}.
\end{align}

\noindent Рассмотрим следующую систему управления:
\begin{equation}\label{eq:2sys}
%\begin{cases}
\left\{\begin{array}{l}
\displaystyle \dot{x} = u_1+u_2-u_3, \\
\displaystyle \dot{y} = u_1-u_2-u_3, \\
\displaystyle \dot{z} = \frac{x-y}{2}u_1-\frac{x+y}{2}u_2+\frac{2-x+y}{2}u_3.
\end{array}\right.
\end{equation}

Запишем общий вид функции Понтрягина
\begin{equation*}\label{eq:3pontHam}
h_{u}^\nu =\langle p, \sum\limits_{i=1}^3 u_i Y_i\rangle + \nu  \sqrt{u_1^2 - u_2^2-u_3^2}, \quad p \in T^*M, \,\, \nu \leq 0.
\end{equation*}

Обозначим $h_i=\langle p,Y_i\rangle.$ Тогда функция Понтрягина выражается следующим образом:
\begin{equation*}
h_{u}^\nu =u_1 h_1+u_2 h_2 +u_3 h_3-\nu  \sqrt{u_1^2 - u_2^2-u_3^2}.
\end{equation*}

В формулировке ПМП, не ограничивая общности, достаточно рассмотреть два случая: $\nu= 0$ — аномальный случай и $\nu =-1$ --- нормальный случай. Рассмотрим их подробно.

\subsection{Аномальный случай принципа максимума Понтрягина}
Пусть $\nu=0$.
Определим $(a,b,c)$ как компоненты ковектора $p$ в канонических координатах. Они изменяются по закону $\displaystyle \dot{p_i}=-\frac{\partial h_{u}^0}{\partial q_i}$:

Систему управления можно записать в виде: \\
\begin{equation*}
\left\{\begin{array}{l}
\displaystyle \dot{x} = u_1+u_2-u_3, \\
\displaystyle \dot{y} = u_1-u_2-u_3, \\
\displaystyle \dot{z} = \frac{x-y}{2}u_1-\frac{x+y}{2}u_2+\frac{2-x+y}{2}u_3,
\end{array}\right.
\quad
\left\{\begin{array}{l}
\displaystyle \dot{a} = (-u_1+u_2+u_3)\frac{c}{2}, \\
\displaystyle \dot{b} = (u_1+u_2-u_3)\frac{c}{2}, \\
\displaystyle \dot{c} = 0.
\end{array}\right.
\end{equation*}\\ 

Связь между $h_i$ и $p_i$ явно записывается как: 
\begin{equation*}
%\begin{cases}
\left\{\begin{array}{l}
\displaystyle h_1 = a+b+c\frac{x-y}{2}, \\
\displaystyle h_2 = a-b-c\frac{x+y}{2}, \\
\displaystyle h_3 = c-h_1. \\
\end{array}\right.
%\end{cases}
\end{equation*}\\ 

$h_i$ суть наперед заданные константы. Опишем пространство $h_i$, разделив его на 4 подпространства.\\
1) $h_2^2+h_3^2-h_1^2 \leq 0$, $h_1 \geq 0$.\\
Запишем ограничения для случая 1
\begin{equation*}
\left\{\begin{array}{l}
\displaystyle u_1^2 \geq u_2^2+u_3^2, \\
\displaystyle h_1^2 \geq h_2^2+h_3^2. \\
\end{array}\right.
\end{equation*}\\  
Тогда справедливы следующие неравенства\\
$u_1^2h_1^2\geq u_2^2h_2^2+h_3^2u_3^2+u_2^2h_3^2+u_3^2h_2^2, $\\
$u_2^2h_2^2+h_3^2u_3^2+u_2^2h_3^2+u_3^2h_2^2 - (u_2h_2+u_3h_3)^2=(u_2h_3-u_3h_2)^2\geq 0.$\\
Откуда следует, что \\
$|u_1h_1|\geq |u_2h_2|+|u_3h_3|.$\\
Исследуем $u_1h_1$. Если $h_1>0$, в совокупности с неограниченностью $u_1$, то и $u_1h_1$ неограниченна.\\
$u_1=0$ означает и $u_2=u_3=0$.\\
Максимума функции $h_u^0$ либо не существует, либо решение тривиально.\\

2) $h_2^2+h_3^2-h_1^2 < 0$, $h_1 < 0.$\\
Почти аналогично 1 случаю. 
\begin{equation*}
\left\{\begin{array}{l}
\displaystyle u_1^2 \geq u_2^2+u_3^2, \\
\displaystyle h_1^2 \geq h_2^2+h_3^2, \\
\end{array}\right.
\end{equation*}\\ 
Только из-за строгого равенства $h_u^0 \neq 0.$ То есть в любом случае функция отрицательна. Максимума нет.

3) $h_2^2+h_3^2-h_1^2 > 0$.\\
При $u = k(\sqrt{h_2^2+h_3^2},h_2, h_3)$ получаем \\
$h_u = k \sqrt{h_2^2+h_3^2}(\sqrt{h_2^2+h_3^2} + h_1) \to + \infty$  при $k \to + \infty$.\\
Поэтому в случае 3) максимума не существует.

4) $h_2^2+h_3^2-h_1^2 = 0$, $h_1 < 0$.\\
$u_1^2=u_2^2+u_3^2.$ \\
%Выбираем натуральную параметризацию $u_3=-1$, $u_1^2=u_2^2+1$.\\
Максимум функции $h_u^0=h_1u_1+ h_2u_2 + h_3u_3=0$ достигается, когда она равна 0. \\
В таком случае управление $(u_1, u_2, u_3) = k (-h_1, h_2, h_3)$. За счет перепараметризации можно считать $k = 1$, т.е. $u_1 = - h_1$, $u_2 =  h_2$, $u_3 = h_3$. Совместно с $c=h_1+h_3$ система примет вид:
\begin{equation*}
\left\{\begin{array}{l}
\displaystyle \dot{x} = -c+2c^2t+h_{20}, \\
\displaystyle \dot{y} = -c-2c^2t-h_{20}, \\
\displaystyle \dot{z} = \frac{x-y}{2}u_1-\frac{x+y}{2}u_2+\frac{2-x+y}{2}u_3,
\end{array}\right.
\quad
\left\{\begin{array}{l}
\displaystyle \dot{h_1} = 2cu_2, \\
\displaystyle \dot{h_2} = 2c(-u_1+u_3), \\
\displaystyle \dot{h_3} = -2cu_2. \\
\end{array}\right.
\end{equation*}

А решение системы примет вид:
\begin{equation*}
\left\{\begin{array}{l}
\displaystyle x = (h_{20}-h_{30}+\sqrt{h^2_{20}+h^2_{30}})t+(\sqrt{h^2_{20}+h^2_{30}}-h_{30})^2t^2, \\
\displaystyle y = (-h_{20}-h_{30}+\sqrt{h^2_{20}+h^2_{30}})t-(\sqrt{h^2_{20}+h^2_{30}}-h_{30})^2t^2, \\
\displaystyle z = h_3t-ch_{20}t^2-\frac{c^3t^3}{3},
\end{array}\right.
\end{equation*}
где $c=-\sqrt{h^2_{20}+h^2_{30}}+h_{30}.$

\begin{figure}
	    \centering
	    \includegraphics[width = 7cm]{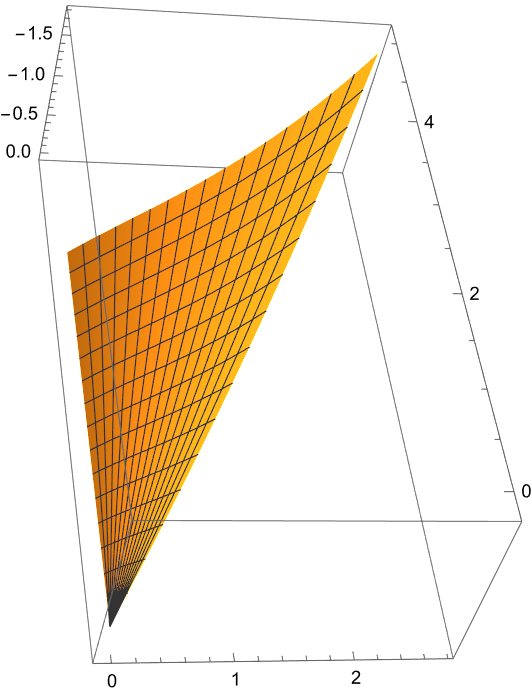}
            \caption{Семейство решений системы с переменными $t\in[0,2], h_1 \in[0,1]$ и зафиксированной константой $h_3=2$.}
\end{figure}

\subsection{Нормальный случай принципа максимума Понтрягина}
Пусть $\nu = -1$.
Как показано в работе \cite{podobryaev}, нормальные траектории удовлетворяют гамильтоновой системе ОДУ с гамильтонианом $H = (-h_1^2 + h_2^2 + h_3^2)/2$  и принадлежат области $h_1 < - \sqrt{h_2^2 + h_3^2}$ (при этом гамильтониан считается чуть иначе). При этом можно считать, что $H \equiv - \frac 12$. $u_1=-h_1, u_2=h_2, u_3=h_3.$ 

Систему управления можно записать в виде: \\
\begin{equation*}
\left\{\begin{array}{l}
\displaystyle \dot{x} = -h_1+h_2-h_3, \\
\displaystyle \dot{y} = -h_1-h_2-h_3, \\
\displaystyle \dot{z} = \frac{-x+y}{2}h_1-\frac{x+y}{2}h_2+\frac{2-x+y}{2}h_3,
\end{array}\right.
\quad
\left\{\begin{array}{l}
\displaystyle \dot{a} = (h_1+h_2+h_3)\frac{c}{2}, \\
\displaystyle \dot{b} = (-h_1+h_2-h_3)\frac{c}{2}, \\
\displaystyle \dot{c} = 0.
\end{array}\right.
\end{equation*}\\ 

Связь между $h_i$ и $p_i$ явно записывается как: \\
\begin{equation*}
%\begin{cases}
\left\{\begin{array}{l}
\displaystyle h_1 = a+b+c\frac{x-y}{2}, \\
\displaystyle h_2 = a-b-c\frac{x+y}{2}, \\
\displaystyle h_3 = -a-b+c\frac{2-x+y}{2}. \\
\end{array}\right.
%\end{cases}
\end{equation*}\\ 

Решение системы примет вид: \\
$$h_{10}=-\sqrt{1+h^2_{20}+h^2_{30}}, c=-\sqrt{1+h^2_{20}+h^2_{30}}+h_{30},$$
\begin{equation*}
\left\{\begin{array}{l}
\displaystyle x = (h_{20}-h_{30}+\sqrt{1+h^2_{20}+h^2_{30}})t+(\sqrt{1+h^2_{20}+h^2_{30}}-h_{30})^2t^2, \\
\displaystyle y = (-h_{20}-h_{30}+\sqrt{1+h^2_{20}+h^2_{30}})t-(\sqrt{1+h^2_{20}+h^2_{30}}-h_{30})^2t^2, \\
\displaystyle z = h_3t-ch_{20}t^2-\frac{c^3t^3}{3}.
\end{array}\right.
\end{equation*}\\ 

\begin{figure}[h]
        \centering
	    \includegraphics[width = 7cm]{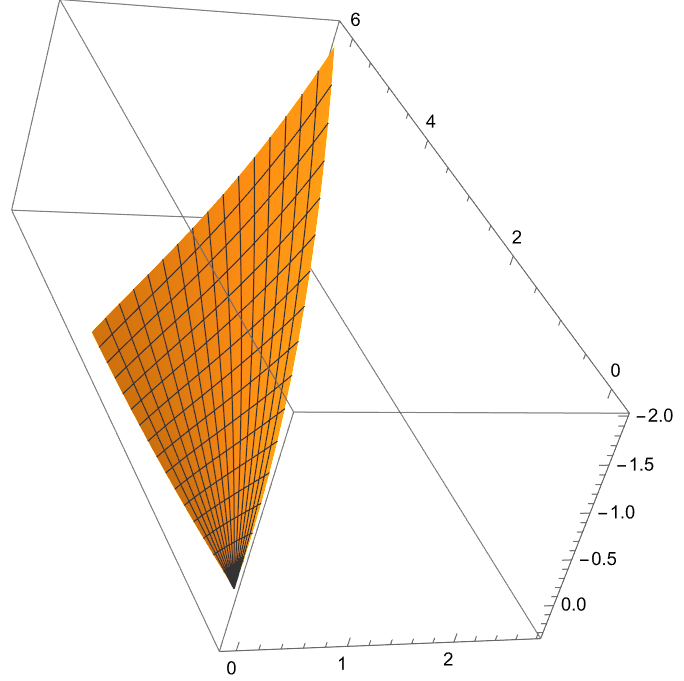}
            \caption{Семейство решений системы с переменными $t\in[0,2], h_2 \in[0,1]$ и зафиксированной константой $h_3=2$.}
\end{figure}

\begin{figure}[h]
        \centering
	    \includegraphics[width = 7cm]{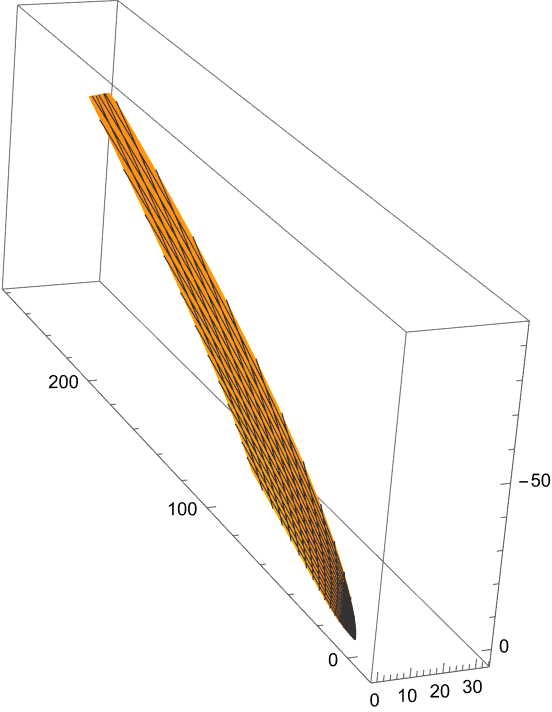}
            \caption{Семейство решений системы с переменными $t\in[0,20], h_2 \in[0,1]$ и зафиксированной константой $h_3=2$.}
\end{figure}

\section{Заключение}
В работе начато исследование трех лоренцевых задач на группе Гейзенберга. К задачам применен принцип максимума Понтрягина, получена параметризация анормальных и нормальных экстремальных траекторий. Исследованы множества достижимости и существование оптимальных траекторий.


\begin{thebibliography}{99}

\bibitem{Muller}
 Muller, O., Sanchez, M. An Invitation to Lorentzian Geometry. Jahresber. Dtsch.
Math.-Ver., 115:3–4 (2014), 153–183.

\bibitem{beem}
Beem, J.K., Ehrlich, P.E., Easley, K.L.: {\em Global Lorentzian Geometry}. Monographs
Textbooks Pure Appl. Math. 202, Marcel Dekker Inc. (1996)

\bibitem{Ivanov}
Иванов, А. О., Тужилин, А. О. Лекции по классической дифференциальной геометрии, Логос, М., 2009.

\bibitem{intro}
Сачков Ю.Л. {\em Введение в геометрическую теорию управления}, М.: URSS, 2021

\bibitem{lor_lob}
Сачков Ю.Л. {\em Лоренцева геометрия на плоскости Лобачевского}. Математические заметки, 2023, том 114, выпуск 1, страницы 154–157.

\bibitem{podobryaev}
А.В.Подобряев. Сублоренцевы экстремали, заданные антинормой //
Дифференциальные уравнения. 60, 3, 386–398 (2024)





\end{thebibliography}
\end{document}